\documentclass[12pt,reqno]{amsart}
\usepackage{etex,mathtools}
\usepackage{amscd,amssymb,amsmath,multicol,tikz,float}

\usepackage[left=1in,right=1in]{geometry}
\begin{document}

\restylefloat{table}
\newtheorem{thm}[equation]{Theorem}
\numberwithin{equation}{section}
\newtheorem{cor}[equation]{Corollary}
\newtheorem{expl}[equation]{Example}
\newtheorem{rmk}[equation]{Remark}
\newtheorem{conv}[equation]{Convention}
\newtheorem{claim}[equation]{Claim}
\newtheorem{lem}[equation]{Lemma}
\newtheorem{sublem}[equation]{Sublemma}
\newtheorem{conj}[equation]{Conjecture}
\newtheorem{defin}[equation]{Definition}
\newtheorem{diag}[equation]{Diagram}
\newtheorem{prop}[equation]{Proposition}
\newtheorem{notation}[equation]{Notation}
\newtheorem{tab}[equation]{Table}
\newtheorem{fig}[equation]{Figure}
\newcounter{bean}
\renewcommand{\theequation}{\thesection.\arabic{equation}}

\raggedbottom \voffset=-.7truein \hoffset=0truein \vsize=8truein
\hsize=6truein \textheight=8truein \textwidth=6truein
\baselineskip=18truept

\def\mapright#1{\ \smash{\mathop{\longrightarrow}\limits^{#1}}\ }
\def\mapleft#1{\smash{\mathop{\longleftarrow}\limits^{#1}}}
\def\mapup#1{\Big\uparrow\rlap{$\vcenter {\hbox {$#1$}}$}}
\def\mapdown#1{\Big\downarrow\rlap{$\vcenter {\hbox {$\ssize{#1}$}}$}}
\def\mapne#1{\nearrow\rlap{$\vcenter {\hbox {$#1$}}$}}
\def\mapse#1{\searrow\rlap{$\vcenter {\hbox {$\ssize{#1}$}}$}}
\def\mapr#1{\smash{\mathop{\rightarrow}\limits^{#1}}}
\def\ss{\smallskip}
\def\s{\sigma}
\def\l{\lambda}
\def\vp{v_1^{-1}\pi}
\def\at{{\widetilde\alpha}}

\def\sm{\wedge}
\def\la{\langle}
\def\ra{\rangle}
\def\lar{\leftarrow}
\def\ev{\text{ev}}
\def\od{\text{od}}
\def\on{\operatorname}
\def\ol#1{\overline{#1}{}}
\def\spin{\on{Spin}}
\def\cat{\on{cat}}
\def\Lbar{\overline{\Lambda}}
\def\qed{\quad\rule{8pt}{8pt}\bigskip}
\def\ssize{\scriptstyle}
\def\a{\alpha}
\def\bz{{\Bbb Z}}
\def\Rhat{\hat{R}}
\def\im{\on{im}}
\def\ct{\widetilde{C}}
\def\ext{\on{Ext}}
\def\sq{\on{Sq}}
\def\eps{\epsilon}
\def\ar#1{\stackrel {#1}{\rightarrow}}
\def\br{{\bold R}}
\def\bC{{\bold C}}
\def\bA{{\bold A}}
\def\bB{{\bold B}}
\def\bD{{\bold D}}
\def\bC{{\bold C}}
\def\bh{{\bold H}}
\def\bQ{{\bold Q}}
\def\bP{{\bold P}}
\def\bx{{\bold x}}
\def\bo{{\bold{bo}}}
\def\dh{\widehat{d}}
\def\A{\mathcal{A}}
\def\B{\mathcal{B}}
\def\si{\sigma}
\def\Vbar{{\overline V}}
\def\dbar{{\overline d}}
\def\wbar{{\overline w}}
\def\Sum{\sum}
\def\tfrac{\textstyle\frac}

\def\tb{\textstyle\binom}
\def\Si{\Sigma}
\def\w{\wedge}
\def\equ{\begin{equation}}
\def\b{\beta}
\def\G{\Gamma}
\def\L{\Lambda}
\def\g{\gamma}
\def\d{\delta}
\def\k{\kappa}
\def\psit{\widetilde{\Psi}}
\def\tht{\widetilde{\Theta}}
\def\psiu{{\underline{\Psi}}}
\def\thu{{\underline{\Theta}}}
\def\aee{A_{\text{ee}}}
\def\aeo{A_{\text{eo}}}
\def\aoo{A_{\text{oo}}}
\def\aoe{A_{\text{oe}}}
\def\vbar{{\overline v}}
\def\endeq{\end{equation}}
\def\sn{S^{2n+1}}
\def\zp{\bold Z_p}
\def\cR{{\mathcal R}}
\def\P{{\mathcal P}}
\def\cQ{{\mathcal Q}}
\def\cj{{\cal J}}
\def\zt{{\bold Z}_2}
\def\bs{{\bold s}}
\def\bof{{\bold f}}
\def\bq{{\bold Q}}
\def\be{{\bold e}}
\def\Hom{\on{Hom}}
\def\ker{\on{ker}}
\def\kot{\widetilde{KO}}
\def\coker{\on{coker}}
\def\da{\downarrow}
\def\colim{\operatornamewithlimits{colim}}
\def\zphat{\bz_2^\wedge}
\def\io{\iota}
\def\om{\omega}
\def\Prod{\prod}
\def\e{{\cal E}}
\def\zlt{\Z_{(2)}}
\def\exp{\on{exp}}
\def\abar{{\overline a}}
\def\xbar{{\overline x}}
\def\ybar{{\overline y}}
\def\zbar{{\overline z}}
\def\mbar{{\overline m}}
\def\nbar{{\overline n}}
\def\sbar{{\overline s}}
\def\kbar{{\overline k}}
\def\bbar{{\overline b}}
\def\et{{\widetilde E}}
\def\ni{\noindent}
\def\tsum{\textstyle \sum}
\def\coef{\on{coef}}
\def\den{\on{den}}
\def\lcm{\on{l.c.m.}}
\def\Ext{\operatorname{Ext}}
\def\iso{\approx}
\def\lra{\longrightarrow}
\def\vi{v_1^{-1}}
\def\ot{\otimes}
\def\psibar{{\overline\psi}}
\def\thbar{{\overline\theta}}
\def\mhat{{\hat m}}
\def\exc{\on{exc}}
\def\ms{\medskip}
\def\ehat{{\hat e}}
\def\etao{{\eta_{\text{od}}}}
\def\etae{{\eta_{\text{ev}}}}
\def\dirlim{\operatornamewithlimits{dirlim}}
\def\gt{\widetilde{L}}
\def\lt{\widetilde{\lambda}}
\def\st{\widetilde{s}}
\def\ft{\widetilde{f}}
\def\sgd{\on{sgd}}
\def\lfl{\lfloor}
\def\rfl{\rfloor}
\def\ord{\on{ord}}
\def\gd{{\on{gd}}}
\def\rk{{{\on{rk}}_2}}
\def\nbar{{\overline{n}}}
\def\MC{\on{MC}}
\def\lg{{\on{lg}}}
\def\cH{\mathcal{H}}
\def\cS{\mathcal{S}}
\def\cP{\mathcal{P}}
\def\N{{\Bbb N}}
\def\Z{{\Bbb Z}}
\def\Q{{\Bbb Q}}
\def\R{{\Bbb R}}
\def\C{{\Bbb C}}
\def\Lb{\overline\Lambda}
\def\mo{\on{mod}}
\def\xt{\times}
\def\notimm{\not\subseteq}
\def\Remark{\noindent{\it  Remark}}
\def\kut{\widetilde{KU}}
\def\Eb{\overline E}
\def\*#1{\mathbf{#1}}
\def\0{$\*0$}
\def\1{$\*1$}
\def\22{$(\*2,\*2)$}
\def\33{$(\*3,\*3)$}
\def\ss{\smallskip}
\def\ssum{\sum\limits}
\def\dsum{\displaystyle\sum}
\def\la{\langle}
\def\ra{\rangle}
\def\on{\operatorname}
\def\proj{\on{proj}}
\def\od{\text{od}}
\def\ev{\text{ev}}
\def\o{\on{o}}
\def\U{\on{U}}
\def\lg{\on{lg}}
\def\a{\alpha}
\def\bz{{\Bbb Z}}
\def\ccM{{\Bbb M}}
\def\E{\mathcal{E}}
\def\eps{\varepsilon}
\def\bc{{\bold C}}
\def\bN{{\bold N}}
\def\bB{{\bold B}}
\def\bW{{\bold W}}
\def\nut{\widetilde{\nu}}
\def\tfrac{\textstyle\frac}
\def\b{\beta}
\def\G{\Gamma}
\def\g{\gamma}
\def\zt{{\Bbb Z}_2}
\def\zth{{\bold Z}_2^\wedge}
\def\bs{{\bold s}}
\def\bx{{\bold x}}
\def\bof{{\bold f}}
\def\bq{{\bold Q}}
\def\be{{\bold e}}
\def\lline{\rule{.6in}{.6pt}}
\def\xb{{\overline x}}
\def\xbar{{\overline x}}
\def\ybar{{\overline y}}
\def\zbar{{\overline z}}
\def\ebar{{\overline \be}}
\def\nbar{{\overline n}}
\def\ubar{{\overline u}}
\def\bbar{{\overline b}}
\def\et{{\widetilde e}}
\def\M{\mathcal{M}}
\def\lf{\lfloor}
\def\rf{\rfloor}
\def\ni{\noindent}
\def\ms{\medskip}
\def\Dhat{{\widehat D}}
\def\what{{\widehat w}}
\def\Yhat{{\widehat Y}}
\def\abar{{\overline{a}}}
\def\minp{\min\nolimits'}
\def\sb{{$\ssize\bullet$}}
\def\mul{\on{mul}}
\def\N{{\Bbb N}}
\def\Z{{\Bbb Z}}
\def\S{\Sigma}
\def\Q{{\Bbb Q}}
\def\R{{\Bbb R}}
\def\C{{\Bbb C}}
\def\Xb{\overline{X}}
\def\eb{\overline{e}}
\def\notint{\cancel\cap}
\def\cS{\mathcal S}
\def\cR{\mathcal R}
\def\el{\ell}
\def\TC{\on{TC}}
\def\GC{\on{GC}}
\def\wgt{\on{wgt}}
\def\Ht{\widetilde{H}}
\def\wbar{\overline w}
\def\dstyle{\displaystyle}
\def\Sq{\on{sq}}
\def\Om{\Omega}
\def\ds{\dstyle}
\def\tz{tikzpicture}
\def\zcl{\on{zcl}}
\def\bd{\bold{d}}
\def\cM{\mathcal{M}}
\def\io{\iota}
\def\Vb#1{{\overline{V_{#1}}}}
\def\Ebar{\overline{E}}
\def\lb{\,\begin{picture}(-1,1)(1,-1)\circle*{3.5}\end{picture}\ }
\def\rlb{\,\begin{picture}(-1,1)(1,-1) \circle*{4.5}\end{picture}\ }
\def\lbb{\,\begin{picture}(-1,1)(1,-1)\circle*{8}\end{picture}\ }
\def\zp{\Z_p}

\title[The connective $KO$-theory of $K(\zt,2)$, I: the $E_2$ page]
{The connective $KO$-theory of the Eilenberg-MacLane space $K(\zt,2)$, I: the $E_2$ page}
\author{Donald M. Davis}
\address{Department of Mathematics, Lehigh University\\Bethlehem, PA 18015, USA}
\email{dmd1@lehigh.edu}
\author{W. Stephen Wilson}
\address{Department of Mathematics, Johns Hopkins University\\Baltimore, MD 01220, USA}
\email{wswilsonmath@gmail.com}
\date{October 30, 2024}

\maketitle
\begin{abstract} We compute the $E_2$ page of the Adams spectral sequence converging to the connective $KO$-theory of the second mod 2 Eilenberg-MacLane space, $ko_*(K(\Z/2,2))$. This required a careful analysis of the structure of $H^*(K(\Z/2,2);\zt)$ as a module over the subalgebra of the Steenrod algebra generated by $\sq^1$ and $\sq^2$. Complete analysis of the spectral sequence will be performed in a subsequent paper.\end{abstract}

\section{Introduction}
Let $\zt=\Z/2$ and let $K_2$ denote the Eilenberg-MacLane space $K(\zt,2)$.
In \cite{DW}, the authors gave a complete determination of the connective complex $K$-theory groups $ku_*(K_2)$ and $ku^*(K_2)$. The original motivation for this work was from \cite{W} and \cite{DWSW},  which studied Stiefel-Whitney classes of Spin manifolds. Because of the relationship (\cite{ABP}) of the Spin cobordism spectrum and the spectrum $ko$ for connective real $K$-theory, information about $ko_*(K(\zt,n))$ gave useful results about Spin manifolds. For complete calculations the authors were led to the more tractable $ku$ groups. In this paper, we return to the $ko$ groups. 

We give a complete determination of the $E_2$ page of the Adams spectral sequences (ASS) converging to  $ko_*(K_2)$ and $ko^*(K_2)$. In a subsequent paper, we will complete the calculation by determining the differentials and extensions in the spectral sequences. We choose to split this $E_2$ work off because we feel that it involves some clever arguments that we would not want to have obscured in a paper with massive ASS charts.

Most of our focus will be on the homology groups $ko_*(K_2)$, in part because of its connection with the motivating problem and in part because its ASS is of a more familiar form than that for $ko^*(-)$. In \cite{DW}, most of the work was done for the cohomology groups $ku^*(K_2)$, largely because of the product structure. That structure, along with a comparison with the mod-$p$ groups $k(1)^*(K_2)$, enabled us to find the differentials in the spectral sequence for $ku^*(K_2)$, and we can use that information to deduce differentials in the other spectral sequences. Similarly to the situation for $ku$ in \cite{DW}, the $ko$-homology and $ko$-cohomology groups of $K_2$ are Pontryagin duals of one another. We discuss this in Section \ref{cohsec}.

Let $A_1$ denote the subalgebra of the mod-2 Steenrod algebra generated by $\sq^1$ and $\sq^2$, and let $E_1$ denote the exterior subalgebra generated by the Milnor primitives $Q_0=\sq^1$ and $Q_1=\sq^1\sq^2+\sq^2\sq^1$. The ASS converging to $ko_*(X)$ has $E_2^{s,t}=\ext^{s,t}_{A_1}(H^*X,\zt)$, while that for $ku_*(X)$ has $E_2=\ext_{E_1}(H^*X,\zt)$. All cohomology groups have coefficients in $\zt$. The first step toward $ku^*(K_2)$ was finding a splitting of $H^*K_2$ as a direct sum of reduced $E_1$-modules (\cite[Proposition 2.11 and (2.16)]{DW}). (A {\it reduced} module is one containing no free submodules.) In Section  \ref{extsec}, we describe a corresponding splitting as $A_1$-modules (Theorem \ref{splt}) and the groups $\ext_{A_1}(-,\zt)$ for all of the summands. This then will be the $E_2$ page of the ASS, the main result of this paper.

\section{The $A_1$-summands $M_k$}\label{Mksec}
An important part of the $E_1$ splitting of $H^*K_2$ was a family of $E_1$-modules $M_k$ for $k\ge4$ (\cite[(2.13), (2.14), (2.15)]{DW}). In this section, we find  corresponding $A_1$-modules, which we also call $M_k$. Although the structure of these $A_1$-modules as $E_1$-modules is very similar to that of the corresponding $E_1$-modules of \cite{DW} (in fact isomorphic if $k\equiv0,1$ mod 4), finding classes with the correct $\sq^2$ behavior was a nontrivial task.

Let $u_0$ denote the nonzero element of $H^2(K_2)$, and define $u_j$ inductively by $u_{j+1}=\sq^{2^j}u_j$.
Then $H^*(K_2)=\zt[u_j: j\ge0]$ with $|u_j|=2^j+1$. Let $S=(\sq^1,\sq^2)$. One easily checks that $$S(u_j)=\begin{cases}(u_1,u_0^2)&j=0\\ (0,u_2)&j=1\\ (u_{j-1}^2,0)&j\ge2.\end{cases}$$
 In Lemma \ref{sq2thm} we replace $u_j$ with generators $x_j$ for $j\ge4$ with similar properties except that $\sq^2\sq^1(x_4)=0$.
 \begin{lem}\label{sq2thm} There are elements $x_j\in H^{2^j+1}(K_2)$ for $j\ge4$ satisfying
 \begin{enumerate}
 \item $x_j\equiv u_j$ mod decomposables,
 \item $\sq^1(x_4)=c_{18}\ne0$, $\sq^2(c_{18})=0$, $\sq^2(x_4)=0$, and
 \item $S(x_j)=(x_{j-1}^2,0)$ for $j\ge5$.
 \end{enumerate}
 \end{lem}
 \begin{proof} We first introduce an intermediate set of generators $w_j$ defined by
 $$w_j=\begin{cases}u_j&j=0,1\\
 u_0u_1+u_2&j=2\\
 u_1^{2^{j-2}}u_{j-2}+u_0^{2^{j-2}} u_{j-1}+u_j&j\ge3.\end{cases}$$
 These satisfy
 $$S(w_j)=\begin{cases}(w_1,w_0^2)&j=0\\
 (0,w_2+w_0w_1)&j=1\\
 (0,w_0w_2)&j=2\\
 (w_{j-1}^2,0)&j\ge3.\end{cases}$$
 Now we define $x_4=w_0w_2^3+w_4$ and, for $j\ge5$
 $$x_j=w_0^{2^{j-4}}w_2^{2^{j-3}}w_{j-2}+w_1^{3\cdot 2^{j-5}}w_2^{2^{j-4}}w_3^{2^{j-5}}w_{j-3}+w_0^{2^{j-5}}w_1^{2^{j-4}}w_2^{2^{j-3}}w_{j-3}+w_j.$$
 One can check that these satisfy the claims of the lemma.\end{proof}

 \begin{thm}\label{Mkthm} For $k\ge4$ there are $Q_0$-free $A_1$-submodules $M_k\subset H^*(K_2)$ with
$$H_*(M_k;Q_1)=\begin{cases}\langle c_{18},x_4\rangle&k=4\\
\langle x_{k-1}^2,c_{18}x_4\prod_{t=4}^{k-2}x_t^2\rangle&k\ge5.\end{cases}$$
The $A_1$-module $\Sigma^{-2^k}M_k$ has the form in Figure \ref{Mkpic}.
\end{thm}

Here and throughout $\langle s_1,\ldots,s_k\rangle$ denotes the span (resp. graded span) of elements in a vector space (resp. graded vector space).
We depict $A_1$-modules with straight segments showing $\sq^1$, and curved segments $\sq^2$.
We circle the $Q_1$-homology classes.

\bigskip
\begin{minipage}{6in}
\begin{fig}\label{Mkpic}

{\bf Modules $\Sigma^{-2^k}M_k$.}

\begin{center}

\begin{\tz}[scale=.4]
\node at (-8,-3) {\lb};
\node at (-6,-3) {\lb};
\node at (-4,-3) {\lb};
\node at (-2,-3) {\lb};

\node at (-11,-3) {$k=4\ell+5$};
\node at (-8,1) {\lb};
\node at (-6,1) {\lb};
\node at (-4,1) {\lb};
\node at (-2,1) {\lb};

\node at (-11,1) {$k=4\ell+4$};
\node at (-8,-6.5) {\lb};
\node at (-6,-6.5) {\lb};
\node at (-4,-6.5) {\lb};
\node at (-2,-6.5) {\lb};

\node at (-11,-6.5) {$k=4\ell+6$};
\node at (-8,-10.5) {\lb};
\node at (-6,-10.5) {\lb};
\node at (-4,-10.5) {\lb};
\node at (-2,-10.5) {\lb};

\node at (-11,-10.5) {$k=4\ell+7$};
\node at (0,-3) {\lb};

\node at (4,-3) {\lb};
\node at (6,-3) {\lb};
\node at (8,-3) {\lb};
\node at (2,-3) {\lb};
\node at (10,-3) {\lb};
\node at (12,-3 ) {\lb};
\node at ( 14,-3) {\lb};
\node at ( 16,-3) {\lb};
\node at ( 18,-3) {\lb};
\node at (20,-3) {\lb};
\node at (22,-3) {\lb};
\draw (20,-3) -- (22,-3);

\draw (0,-3) -- (2,-3);
\draw (4,-3) -- (6,-3);
\node at (9,-3) {$\ldots$};
\draw (12,-3) -- (14,-3);
\draw (16,-3) -- (18,-3);

\draw (4,-3) to[out=330, in=210] (8,-3);
\draw (12,-3) to[out=30, in=150] (16,-3);
\draw (18,-3) to[out=30, in=150] (22,-3);
\node at (-8,-3.7) {$1$};
\node at (20,-3.7) {$8\ell+3$};

\draw (2,-3) to[out=30, in=150] (6,-3);
\draw (10,-3) to[out=330, in=210] (14,-3);


\node at (0,-10.5) {\lb};

\node at (4,-10.5) {\lb};
\node at (6,-10.5) {\lb};
\node at (8,-10.5) {\lb};
\node at (2,-10.5) {\lb};
\node at (16,-10.5) {\lb};
\node at (20,-10.5) {\lb};
\node at (22,-12 ) {\lb};
\node at ( 20,-12) {\lb};
\node at ( 18,-12) {\lb};
\node at ( 14,-12) {\lb};
\node at (16,-12) {\lb};
\draw (0,-10.5) -- (2,-10.5);
\draw (4,-10.5) -- (6,-10.5);
\node at (9,-10.5) {$\ldots$};
\draw (16,-12) -- (14,-12);
\draw (22,-12) -- (20,-12);

\draw (4,-10.5) to[out=330, in=210] (8,-10.5);
\draw (2,-10.5) to[out=30, in=150] (6,-10.5);
\draw (18,-12) -- (20,-10.5);
\node at (10,-10.5) {\lb};
\node at (14,-10.5) {\lb};
\draw (14,-10.5) -- (16,-10.5);
\draw (10,-10.5) to[out=30, in=150] (14,-10.5);
\node at (12,-10.5) {\lb};
\draw (10,-10.5) -- (12,-10.5);
\node at (-8,-11.2) {$1$};
\node at (18,-10.9) {$8\ell+7$};

\draw (16,-12) to[out=330, in=210] (20,-12);
\draw (16,-10.5) to[out=30, in=150] (20,-10.5);
\draw (14,-12) to[out=30, in=150] (18,-12);
\draw (18,-12) to[out=30, in=150] (22,-12);


\node at (0,1) {\lb};

\node at (4,1) {\lb};
\node at (6,1) {\lb};
\node at (8,1) {\lb};
\node at (2,1) {\lb};
\node at (10,1) {\lb};
\node at (12,1) {\lb};
\node at ( 14,1) {\lb};
\node at ( 16,1) {\lb};
\node at ( 18,1) {\lb};
\draw (0,1) -- (2,1);
\draw (4,1) -- (6,1);
\draw (12,1) -- (14,1);
\draw (16,1) -- (18,1);
\node at (9,1) {$\cdots$};
\draw (4,1) to[out=330, in=210] (8,1);
\draw (12,1) to[out=30, in=150] (16,1);
\node at (-8,.3) {$1$};
\node at (16,.3) {$8\ell+1$};

\draw (2,1) to[out=30, in=150] (6,1);
\draw (10,1) to[out=330, in=210] (14,1);
\draw (-8,1) -- (-6,1);
\draw (-4,1) -- (-2,1);
\draw (-8,-3) -- (-6,-3);
\draw (-4,-3) -- (-2,-3);
\draw (-8,-6.5) -- (-6,-6.5);
\draw (-4,-6.5) -- (-2,-6.5);
\draw (-8,-10.5) -- (-6,-10.5);
\draw (-4,-10.5) -- (-2,-10.5);
\draw (-4,1) to[out=30, in=150] (0,1);
\draw (-4,-3) to[out=30, in=150] (0,-3);
\draw (-4,-6.5) to[out=30, in=150] (0,-6.5);
\draw (-4,-10.5) to[out=30, in=150] (0,-10.5);
\draw (-6,1) to[out=330, in=210] (-2,1);
\draw (-6,-3) to[out=330, in=210] (-2,-3);
\draw (-6,-6.5) to[out=330, in=210] (-2,-6.5);
\draw (-6,-10.5) to[out=330, in=210] (-2,-10.5);
\draw (-6,1) circle [radius=.45];
\draw (-6,-3) circle [radius=.45];
\draw (-6,-6.5) circle [radius=.45];
\draw (-6,-10.5) circle [radius=.45];
\draw (16,1) circle [radius=.45];
\draw (20,-3) circle [radius=.45];
\draw (18,-12) circle [radius=.45];

\node at (0,-6.5) {\lb};

\node at (4,-6.5) {\lb};
\node at (6,-6.5) {\lb};
\node at (8,-6.5) {\lb};
\node at (2,-6.5) {\lb};
\node at (12,-6.5) {\lb};
\node at (14,-8 ) {\lb};
\node at ( 22,-8) {\lb};
\node at ( 20,-8) {\lb};
\node at ( 16,-8) {\lb};
\node at (12,-6.5) {\lb};
\node at (14,-6.5) {\lb};
\node at (16,-6.5) {\lb};

\draw (0,-6.5) -- (2,-6.5);
\draw (4,-6.5) -- (6,-6.5);
\node at (9,-6.5) {$\ldots$};
\draw (14,-8) -- (16,-8);
\draw (20,-8) -- (22,-8);

\draw (4,-6.5) to[out=330, in=210] (8,-6.5);
\draw (2,-6.5) to[out=30, in=150] (6,-6.5);
\draw (10,-6.5) to[out=330, in=210] (14,-6.5);
\draw (12,-6.5) to[out=30, in=150] (16,-6.5);

\draw (16,-8) to[out=330, in=210] (20,-8);
\draw (14,-8) to[out=45, in=225] (18,-6.5);
\draw (18,-6.5) to[out=315, in=135] (22,-8);
\node at (18,-6.5) {\lb};
\draw (16,-6.5) -- (18,-6.5);
\node at (10,-6.5) {\lb};
\draw (12,-6.5) -- (14,-6.5);
\node at (-8,-7.2) {$1$};
\node at (16,-5.5) {$8\ell+5$};
\draw (16,-7.25) circle [x radius=.6,y radius=1.2];

\end{\tz}
\end{center}
\end{fig}
\end{minipage}

\bigskip
For example, if $k=4\ell+4$, $\Sigma^{-2^k}M_k$ has a single nonzero class $g_i$ for $1\le i\le 8\ell+2$ with 
$$\sq^2\sq^1\sq^2(g_i)=\sq^1g_{i+4}\ne0\text{ if }i\equiv 3\ (4), \ i\le 8\ell-5,$$
and $\sq^2\sq^1(g_1)=\sq^1(g_3)\ne0$.
\begin{proof}[Proof of Theorem \ref{Mkthm}] We use the classes $x_j$, $j\ge4$, of Lemma \ref{sq2thm}, but find it convenient to write $c_{18}$ as $x_3^2$, even though it isn't a perfect square. In the discussion below we treat it as a perfect square. For $k\ge4$, let $\ccM_k$ denote the finite $A_1$-submodule of $H^*K_2$ with basis all elements $\ds\prod_{j=3}^kx_j^{e_j}$ satisfying $\sum e_j2^{j}=2^{k}$. Our desired $A_1$-module $M_k$ will be a submodule of $\ccM_k$.

We first show that $\ccM_k$ is $Q_0$-free. Every monomial in $\ccM_k$ which is a perfect square can be written uniquely as $\ds\prod_{s\in S}x_s^2\cdot\prod_{t\in T}x_t^{2e_t}$ with $e_t>1$ and $S$ and $T$ disjoint. It determines a $Q_0$-free summand
$$\prod_{t\in T}x_t^{2e_t-2}\bigotimes_{i\in S\cup T}\langle x_{i+1},x_i^2\rangle.$$ Every monomial in $\ccM_k$ is in a unique one of these summands, as can be seen by writing the monomial as $\ds P\cdot\prod_{u\in U}x_u$ with $P$ a perfect square. This monomial is in the $Q_0$-free summand determined as above from $\ds P\cdot\prod_{u\in U}x_{u-1}^2$.

We now show, somewhat similarly, that $$H_*(\ccM_k;Q_1)=
\langle x_{k-1}^2,x_4\prod_{t=3}^{k-2}x_t^2\rangle.$$
Let $k\ge5$, as the case $k=4$ is elementary. Every monomial in $\ccM_k$ which is a perfect square or $x_4$ times a perfect square can be written uniquely as $\ds\prod_{s\in S}x_s^2\cdot x_4^\eps\cdot\prod_{t\in T}x_t^{2e_t}$ with $e_t\ge2$, $S$ and $T$ disjoint, and $\eps\in\{0,1\}$. Also, $T\ne\emptyset$ unless the monomial is $x_{k-1}^2$ or $x_3^2x_4^3x_5^2\cdots x_{k-2}^2$, in order to have $\sum e_j2^j=2^k$. This monomial determines a $Q_1$-free summand 
$$\prod_{s\in S}x_s^2\cdot x_4^\eps\cdot\prod_{t\in T}x_t^{2e_t-4}\bigotimes_{t\in T}\langle x_{t}^4,x_{t+2}\rangle.$$
 Every monomial in $\ccM_k$ except $x_{k-1}^2$ and $x_3^2x_4^3x_5^2\cdots x_{k-2}^2$ is in a unique one of these by writing it as $$P\cdot x_4^\eps\cdot \prod_{\substack{t\in T\\ t>2}}x_{t+2}$$ with $P$ a perfect square; it is in the $Q_0$-free summand determined as above from $\ds P\cdot x_4^\eps\prod_{t\in T}x_{t}^4$.

By \cite[Proposition 13.13 and p.203]{Mar}, the $A_1$-module $\ccM_k$ has an expression, unique up to isomorphism, as $M_k\oplus F$, with  $F$ free and $M_k$ reduced. This $M_k$ is $Q_0$-free and has the $Q_1$-homology stated in the theorem. 
To get a sense of why this is true, it is impossible for a $Q_0$-free module to have just one $Q_1$-homology class.  Thus the two $Q_1$-homology classes must be in the same summand and what is left must be free over $A_1$.

We will determine its precise structure.

The module $\ccM_4$ has only the classes $\langle x_4,x_3^2\rangle$, so this is also $M_4$. For $k\ge5$, $\ccM_k$ in gradings $\le 2^k+4$ has just the classes $\langle x_k, x_{k-1}^2, x_{k-2}^2x_{k-1}, x_{k-2}^4\rangle$, in which $\sq^1$ and $\sq^2$ act as depicted on the left four dots in each row of Figure \ref{Mkpic}. We will use Yu's Theorem (\cite[Theorem 7.1]{Br}) to show that $M_k$ must have the form claimed in the theorem. We thank Bob Bruner for suggesting the use of Yu's Theorem.

 For $k\ge5$, let $\ccM_k^*$ denote the $A_1$-module dual to $\ccM_k$. Its top class $x_k^*$ is in grading $-2^k-1$ and bottom class $(x_3^{2^{k-3}})^*$ is in grading $-2^k-2^{k-3}$. Let $(\ccM_k^*)^+$ denote an $A_1$-module which agrees with $\ccM_k^*$ in gradings less than $-2^k$ and for $i\ge -2^k$ has a single nonzero class $y_i$ in grading $i$, with $\sq^2\sq^3y_{4j}=y_{4j+5}=\sq^1y_{4j+4}$, and $0\ne\sq^1y_{-2^k}\in\im(\sq^2)$. 
 This $(\ccM_k^*)^+$ is $Q_0$-free and has a single nonzero $Q_1$-homology class, dual to $x_4\prod_{t=3}^{k-2}x_t^2$, in grading $7-2^k-2k$. By \cite[Proposition 13.13 and p.203]{Mar}, $(\cM_k^*)^+$ is isomorphic to the direct sum of a reduced module $R$ and a free module. Since $R$ is $Q_0$-free and reduced with a single nonzero $Q_1$-homology class, by Yu's Theorem, $R$
 is  isomorphic to a shifted version of one of the four modules $P_i$, $0\le i\le3$, depicted in \cite[Figure 1]{Br}. These modules begin with a form dual to one of the four endings of the modules in Figure \ref{Mkpic}, followed by an infinite string of $\sq^1z_n=\sq^2\sq^3 z_{n-4}$.

 Our module $M_k$ is defined as the dual of $R/T$, where $T$ is the submodule of $R$ consisting of classes of grading $\ge-2^k$. This $M_k$ will begin the same way as $\ccM_k$, as $\Sigma^{2^k}\langle g_1, \sq^1g_1, g_3, \sq^1g_3=\sq^2\sq^1g_1\rangle$, and will end with  one of the four types in Figure \ref{Mkpic}, although {\it a priori} it could have a different length. Its top $Q_1$-homology class is in grading $2^k+2k-7$.
 
 Since $A_1$ has 8 basis elements, the total number of basis elements in $M_k$ will be congruent mod 8 to the number in $\ccM_k$. There is a 1-1 correspondence between a basis for $\ccM_k$ and the set of partitions of $2^{k-3}$ into 2-powers. ($e_j$ tells the number of occurrences of $2^{j-3}$.) It is proved in \cite{Ch} that this number of partitions is $\equiv2$ mod 8 if $k$ is even, and is $\equiv4$ mod 8 if $k$ is odd. 

 Let $k=4\ell+4$. The first module in Figure \ref{Mkpic} is the only possibility that satisfies that the top $Q_1$-homology class is in grading $2^k+2k-7=2^k+8\ell+1$ and the number of basis elements is $\equiv2$ mod 8. The second and fourth types in Figure \ref{Mkpic} have their top $Q_1$-homology class in grading 3 mod 4, while if the third type had its top $Q_1$-homology class in $2^k+8\ell+1$, its number of basis elements would be 6 mod 8. A similar analysis, utilizing top $Q_1$-homology class mod 4 and number of basis elements mod 8, shows the $M_k$ must be as claimed.

\end{proof}

Prior to discovering this proof, we had laboriously found explicit bases for $M_k$ for $k\le9$. For example, with $abcd$ denoting $x_6^ax_5^bx_4^cc_{18}^d$, the basis for $M_7$ had $x_7$, $2000$, $1200$, $0400$, $1040$, $0240$, $0080$ along the top, as pictured in Figure \ref{Mkpic}, and $1111+0320$, $0311+1031+1102+0240$, $0231+1022+0160$, $0151+1013+0222+0080$, and $0071+0142+0213+1004$ along the bottom.

\section{Ext charts and tensor products}\label{extsec}

There is a nice pattern to the charts $\ext_{A_1}^{s,t}(\Sigma^{-2^k}M_k,\zt)$, depicted, as usual, in coordinates $(t-s,s)$. They are similar to familiar charts of $\ext_{A_1}(H^*P^{2n},\zt)$ (e.g., \cite{D}). In fact, there are $A_1$-module isomorphisms $\Sigma^{-2^{4\ell+4}}M_{4\ell+4}\approx H^*P^{8\ell+2}$ and $\Sigma^{-2^{4\ell+5}}M_{4\ell+5}\approx H^*P^{8\ell+4}$. For all $k$, all classes in these charts are $v_1^4$-periodic; i.e., $\ext^{s,t}\to\ext^{s+4,t+12}$ is bijective for $s\ge0$.  All the charts have the same upper edge, $(8i+x,4i+y)$ for $(x,y)=(1,0)$, $(2,1)$, $(3,2)$, and $(7,3)$. The lower edge drops by 1 for each increase in $k$, as long as $s\ge0$. In Figure \ref{Mcharts} we show the beginning of the charts for $5\le k\le7$. These Ext charts are easily obtained by standard methods from the explicit description of the modules in Figure \ref{Mkpic}. See \cite[Appendix A]{BG} for a rather detailed discussion of these methods.

\bigskip
\begin{minipage}{6in}
\begin{fig}\label{Mcharts}

{\bf $\ext_{A_1}(\Sigma^{-2^k}M_k,\zt)$.}

\begin{center}

\begin{\tz}[scale=.4]
\draw (-.5,0) -- (6.5,0);
\draw (8.5,0) -- (21.5,0);
\draw (23.5,0) -- (36.5,0);
\node at (0,0) {\lb};
\draw (0,0) -- (2,2) -- (2,0) -- (4,2);
\node at (1,1) {\lb};
\node at (2,2) {\lb};
\node at (2,0) {\lb};
\node at (2,1) {\lb};
\node at (3,1) {\lb};
\node at (4,2) {\lb};
\node at (6,3) {\lb};
\node at (9,0) {\lb};
\node at (10,1) {\lb};
\node at (11,2) {\lb};
\node at (11,0) {\lb};\node at (11,1) {\lb};
\node at (12,0) {\lb};
\draw (9,0) -- (11,2) -- (11,0);
\node at (0,-.8) {$1$};
\node at (6,-.8) {$7$};
\node at (3,-2) {$k=5$};
\node at (13,1) {\lb};
\node at (15,2) {\lb};
\node at (15,3) {\lb};
\node at (19,3) {\lb};
\node at (19,4) {\lb};
\node at (19,5) {\lb};
\node at (19,6) {\lb};
\node at (20,4) {\lb};
\draw (15,2) -- (15,3);
\draw (17,4) -- (19,6) -- (19,3) -- (21,5);
\node at (9,-.8) {$1$};
\node at (15,-.8) {$7$};
\node at (19,-.8) {$11$};
\node at (15,-2) {$k=6$};
\node at (17,4) {\lb};
\node at (18,5) {\lb};
\node at (21,5) {\lb};
\node at (24,0) {\lb};
\node at (25,1) {\lb};
\node at (26,2) {\lb};
\node at (26,1) {\lb};
\node at (26,0) {\lb};
\node at (28,0) {\lb};
\node at (30,1) {\lb};
\draw (24,0) -- (26,2) -- (26,0);
\draw (30,1) -- (30,3);
\draw (32,4) -- (34,6) -- (34,2) -- (36,4);
\node at (24,-.8) {$1$};
\node at (30,-.8) {$7$};
\node at (34,-.8) {$11$};
\node at (10,1) {\lb};
\node at (30,2) {\lb};
\node at (30,3) {\lb};
\node at (32,4) {\lb};
\node at (33,5) {\lb};
\node at (34,6) {\lb};
\node at (34,5) {\lb};
\node at (34,4) {\lb};
\node at (34,3) {\lb};
\node at (34,2) {\lb};
\node at (35,3) {\lb};
\node at (36,4) {\lb};
\draw (12,0) -- (13,1);
\node at (30,-2) {$k=7$};

\end{\tz}
\end{center}
\end{fig}
\end{minipage}
\bigskip

Explicitly, $\Sigma^{-2^k}M_k$ has, for $i\ge0$,
\begin{itemize}
    \item $0$ in $8i+6,8$,
    \item $\zt$ in $8i+1,2$ of filtration $4i+0,1$,
    \item $\zt$ in $8i+4,5$ of filtration $4i-k+6,7$ if $4i-k+6,7\ge0$, else 0,
    \item $\Z/2^{k-4}$ in $8i+7$ with generator of filtration $4i-k+8$ if $4i-k+8\ge0$, else $\Z/2^{4i+4}$ with generator of filtration 0, and
    \item $\Z/2^{k-2}$ in $8i+3$ with generator of filtration $4i-k+5$ if $4i-k+5\ge0$, else $\Z/2^{4i+3}$ with generator of filtration 0.
\end{itemize}
\ni Here, as usual, $d$ dots connected by vertical segments yield a $\Z/2^d$ summand.

The $A_1$-modules $M_k$ in Section \ref{Mksec} correspond to the $E_1$-modules $M_k$ in the $E_1$-splitting of $H^*(K_2)$ in \cite[(2.16)]{DW}. The correspondence is that, as an $E_1$-module, the $A_1$-module $M_k$ is isomorphic to the $E_1$-module $M_k$ plus perhaps a single copy of $E_1$. Moreover, the $Q_1$-homology classes agree, with $u_j$ replaced by $x_j$. Also involved in the $E_1$ splitting in \cite[(2.16)]{DW} were summands $M_k\cdot P$, where $P$ is a product of finitely many distinct classes $u_j^2$ with $j\ge k$. Although $u_j^2$ is acted on trivially by $E_1$, $\sq^2(u_j^2)\ne0$, so the corresponding $A_1$ summands must do more than just multiply by the product of the classes $u_j^2$. To maintain some consistency with \cite{DW}, in Definition \ref{Mzdef} we will define $M_kz_j$ to be a reduced $Q_0$-free $A_1$-module with
\begin{equation}\label{Mzd}H_*(M_kz_j;Q_1)=H_*(M_k;Q_1)\ot\langle u_{j+1}^2\rangle,\end{equation}
and similarly for products with more than one $z_j$.

For $j\ge3$, let $G_j=\langle u_{j+2},u_{j+1}^2,u_j^4\rangle$ with $\sq^2\sq^1u_{j+2}=u_j^4$. If $M$ is a $Q_0$-free $A_1$-module, then $M\ot G_j$ is $Q_0$-free and 
$$H_*(M\ot G_j;Q_1)=H_*(M;Q_1)\ot\langle u_{j+1}^2\rangle.$$
\begin{defin}\label{Mzdef}
We define $M_kz_j$ to be the reduced summand of the $A_1$-module $M_k\ot G_j$. \end{defin}
Let $P_j$ be the $A_1$-module for which there is a short exact sequence (SES)
$$0\to G_j\to P_j \to \Sigma^{2^{j+2}-1}\zt\to 0$$
with $u_{j+2}\in\im(\sq^2)$. Then $H_*(P_j;Q_1)=0$, so $M_k\ot P_j$ is a free $A_1$-module by Wall's Theorem (\cite{Wall}), using also a K\"unneth Theorem for $Q_i$-homology. The short exact sequence of $A_1$-modules
\begin{equation}\label{s}0\to M_k\ot G_j\to M_k\ot P_j \to \Sigma^{2^{j+2}-1}M_k\to 0\end{equation}
has a long exact sequence Ext sequence which implies that 
$$\ext^{s,t}_{A_1}(M_k\ot G_j,\zt)\to \ext_{A_1}^{s+1,t+1}(\Sigma^{2^{j+2}}M_k,\zt)$$
is bijective for $s\ge1$ and surjective for $s=0$. We deduce that, for the reduced submodule, $\ext_{A_1}(M_kz_j,\zt)$ is formed from $\ext_{A_1}(\Sigma^{2^{j+2}}M_k,\zt)$ by shifting filtrations down by 1, or, equivalently, by killing classes of filtration 0. Elements in the kernel of (\ref{s}) when $s=0$ correspond to free summands, which do not appear in the reduced submodule.
Iterating, we have 
\begin{prop} For distinct $j_i\ge k-1$, $\ext_{A_1}(M_kz_{j_1}\cdots z_{j_r},\zt)$ is formed from\linebreak $\ext_{A_1}(\Sigma^{2^{{j_1}+2}+\cdots +2^{{j_r}+2}}M_k,\zt)$ by reducing filtrations by $r$.\label{Mzprop}\end{prop}

The $E_1$-splitting of $H^*K_2$ in \cite[Proposition 2.11]{DW} also involved products of modules with a class called $u_2^2$ there, but would be $u_0^2$ in our notation. Again, since $\sq^2(u_0^2)\ne0$, we must expand to an $A_1$-submodule of $H^*(K_2)$, namely
\begin{equation}\label{Udef}U=\langle u_0, u_1, u_0^2,u_2,u_1^2\rangle.\end{equation}
The $A_1$-structure of this is $\Sigma^2\langle 1,\sq^1,\sq^2, \sq^2\sq^1,\sq^3\sq^1\rangle$, sometimes called the Joker (\cite{Bak}). Note that $H_*(U;Q_1)=\langle u_0^2\rangle$.

\begin{prop}\label{Uext} If $M$ is a $Q_0$-free $A_1$-module and $U$ is as above, then for $s>0$
$$\ext_{A_1}^{s,t}(U\ot M,\zt)\approx \ext_{A_1}^{s+2,t+2}(M,\zt).$$
\end{prop}
\begin{proof} There is a SES of $A_1$-modules
$$0\to G\to F\to U\to 0,$$
where $F$ is a free $A_1$-module on a generator of degree 2, and $G=\langle \io_5,\sq^2\io_5,\sq^3\io_5\rangle$. After tensoring with $M$, the exact Ext sequence yields an isomorphism for $s>0$
$$\ext_{A_1}^{s,t}(G\ot M,\zt)\to \ext_{A_1}^{s+1,t}(U\ot M,\zt).$$
Let $P=\Sigma^5A_1/(\sq^1)$. There is a SES of $A_1$-modules
$$0\to \Sigma^{10}\zt\to P\to G\to 0.$$
Then $P\ot M$ is free by Wall's theorem, since $H_*(P;Q_1)=0$ and $H_*(M;Q_0)=0$. So tensoring this sequence with $M$ yields isomorphisms for $s>0$
$$\ext_{A_1}^{s,t}(\Sigma^{10}M,\zt)\to \ext^{s+1,t}_{A_1}(G\ot M,\zt).$$
Combining the two yields
$$\ext_{A_1}^{s,t}(\Sigma^{10}M,\zt)\approx \ext_{A_1}^{s+2,t}(U\ot M,\zt).$$
The $Q_0$-free module $U\ot M$ has $v_1^4$-periodicity in Ext
$$\ext_{A_1}^{s,t}(U\ot M,\zt)\approx \ext_{A_1}^{s+4,t+12}(U\ot M,\zt)$$
for $s>0$ by \cite[Theorem 5.1]{JFA}. This is isomorphic to $\ext_{A_1}^{s+2,t+12}(\Sigma^{10}M,\zt)\approx \ext_{A_1}^{s+2,t+2}(M,\zt)$.

\end{proof}

We let $UM_k$ and $UM_kz_{j_1}\cdots z_{j_r}$ denote reduced modules after tensoring with $U$. By Proposition \ref{Uext}, their Ext charts are obtained from those of $M_k$ or $M_kz_{j_1}\cdots z_{j_r}$ by decreasing filtrations by 2.

The summand $S$ in \cite[Proposition 2.11]{DW} is the reduced summand of tensor products of the summands of the type that we have been considering here with an $E_1$-module $N$ with $Q_1$-homology  class $x_9$. We have an analogous construction in the $A_1$ context.

Using the classes $w_j$ in the proof of Theorem \ref{sq2thm}, let $N$ be the $A_1$-module $$N=\langle w_2,w_0w_2,w_1w_2,w_3,w_2^2\rangle.$$
This satisfies $\sq^2\sq^3(w_2)=\sq^1(w_3)=w_2^2$ with $|w_2|=5$. It has the property that if $M$ is a $Q_0$-free $A_1$-module, then \begin{equation}\label{Niso}\ext_{A_1}(N\ot M,\zt)\approx \ext_{A_1}(\Sigma^9M,\zt)\end{equation} in positive filtration as is easily seen from the Ext sequence obtained from the SES
$$0\to N'\ot M\to N\ot M\to\Sigma^9M\to 0,$$
where $N'$ is the $A_1$-submodule of $N$ generated by $w_2$, since $N/N'=\Sigma^9\zt$ and $N'\ot M$ is a free $A_1$-module by Wall's theorem. For any of our modules $U^\eps M_k z_J$,
we let $NU^\eps M_kz_J$ denote a reduced submodule of $N\ot U^\eps M_kz_J$. It is isomorphic to $\Sigma^9U^\eps M_kz_J$.

The analogue of \cite[Proposition 2.11]{DW} is given in Theorem \ref{splt}. We let $y_1^2=u_0^4$; it is annihilated by $\sq^1$ and $\sq^2$.
\begin{thm}\label{splt} There is an $A_1$-module splitting
$$H^*K_2=P[y_1^2]\ot(\zt\oplus U\oplus N\oplus NU)\ot(\zt\oplus\bigoplus_{k\ge4}M_k\L_{k-1})\oplus F,$$
where $F$ is free and $\L_{k-1}=E[z_j:j\ge k-1]$ is an exterior algebra. The interpretation of $M_kz_{j_1}\cdots z_{j_r}$ is as in Definition \ref{Mzdef}, and $U\ot M_k\L_{k-1}$, $N\ot M_k\L_{k-1}$, and $NU\ot M_k\L_{k-1}$ mean the reduced summand. For reduced cohmology, one can remove the  $\zt$ summand from the splitting.\end{thm}
Theorem \ref{splt} is obtained from \cite[Proposition 2.11]{DW} by modifying the $E_1$ summands (where necessary) to make them $A_1$ modules that still retain the same $Q_1$ and $Q_0$ homologies.  
\begin{proof} The correspondence with \cite[Proposition 2.11]{DW} is $R\leftrightarrow \bigoplus M_k\L_{k-1}$, $S=NR$, $\langle u_2^2\rangle\leftrightarrow U$, and $P[u_2^2]\leftrightarrow P[y_1^2]\ot (\zt\oplus U)$. The $Q_0$- and $Q_1$-homology classes correspond and fill out the $Q_i$-homology of $H^*K_2$. The quotient of $H^*K_2$ by this large submodule is $A_1$-free by Wall's theorem.\end{proof}

The $E_2$ page is obtained by applying $\ext_{A_1}(-,\zt)$ to the summands of Theorem \ref{splt}. Earlier in this section, we have done that for the summands involving $M_k\L_{k-1}$. 
The others are small modules whose Ext is easily seen to be as in Figure \ref{4ext}.

\bigskip
\begin{minipage}{6in}
\begin{fig}\label{4ext}

{\bf $\ext_{A_1}(-,\zt)$.}

\begin{center}

\begin{\tz}[scale=.32]
\draw (-1,0) -- (6,0);
\draw (2,2) -- (0,0) -- (0,7);
\draw (4,3) -- (4,7);
\node at (6,5) {$\cdot$};
\node at (5.6,4.6) {$\cdot$};
\node at (6.4,5.4) {$\cdot$};
\draw (8,0) -- (17.5,0);
\node at (9,0) {\lb};
\draw (11,1) -- (11,7);
\draw (17,4) -- (15,2) -- (15,7);
\draw (20,0) -- (31,0);
\draw (21,0) -- (21,7);
\draw (27,2) -- (25,0) -- (25,7);
\draw (29,3) -- (29,7);
\draw (33,0) -- (45,0);
\draw (34,0) -- (34,7);
\node at (36,0) {\lb};
\draw (38,1) -- (38,7);
\draw (44,4) -- (42,2) -- (42,7);
\node at (0,-.6) {$8$};
\node at (4,-.6) {$12$};
\node at (9,-.6) {$2$};
\node at (11,-.6) {$4$};
\node at (15,-.6) {$8$};
\node at (21,-.6) {$5$};
\node at (25,-.6) {$9$};
\node at (34,-.6) {$9$};
\node at (38,-.6) {$13$};
\node at (42,-.6) {$17$};
\node at (2,-1.6) {$\langle y_1^4\rangle$};
\node at (12,-1.6) {$U$};
\node at (24,-1.6) {$N$};
\node at (37,-1.6) {$NU$};
\node at (17.4,5) {$\cdot$};
\node at (17.8, 5.4) {$\cdot$};
\node at (18.2,5.8) {$\cdot$};
\node at (29.4,5) {$\cdot$};
\node at (29.8,5.4) {$\cdot$};
\node at (30.2,5.8) {$\cdot$};
\node at (44,6) {$\cdot$};
\node at (44.4,6.4) {$\cdot$};
\node at (44.8,6.8) {$\cdot$};

\end{\tz}
\end{center}
\end{fig}
\end{minipage}
\bigskip

\ni Here $NU$ means the reduced summand of the $A_1$-module $N\ot U$.

\section{$ko$-cohomology and duality}\label{cohsec}
Our main focus is on $ko_*(K_2)$, in part because of its relationship with Spin cobordism. In this short section, we explain briefly how we compute $ko^*(K_2)$ and  the duality between it and $ko_*(K_2)$.

The Adams spectral sequence for $ko^{-*}(K_2)$ is obtained by applying $\ext_{A_1}(\zt,-)$ to the same $A_1$-modules used for $ko_*(K_2)$, with corresponding differentials.  As we did for $ku$ in \cite{DW}, we display the $ko$-cohomology groups increasing from right to left. 

In Figure \ref{Mcoh}, we show the beginning of the charts for $\ext_{A_1}(\zt,M_k)$ for $k=5,6,7$. This should be enough to suggest the entire pattern. These charts are the analogue of those in Figure \ref{Mcharts}. They can be easily obtained from Figure \ref{Mkpic}.

\bigskip
\begin{minipage}{6in}
\begin{fig}\label{Mcoh}

{\bf $\ext_{A_1}(\zt,\Sigma^{-2^k}M_k)$}

\begin{center}

\begin{\tz}[scale=.4]
\draw (-.5,0) -- (6.5,0);
\draw (8.5,0) -- (19.5,0);
\draw (21.5,0) -- (33.5,0);
\node at (0,0) {\lb};
\node at (2,1) {\lb};
\node at (3,2) {\lb};
\node at (4,3) {\lb};
\node at (4,2) {\lb};
\node at (4,1) {\lb};
\node at (5,2) {\lb};
\node at (6,3) {\lb};
\node at (9,0) {\lb};
\node at (13,0) {\lb};
\node at (13,1) {\lb};
\node at (15,2) {\lb};
\node at (16,3) {\lb};
\node at (17,4) {\lb};
\node at (17,3) {\lb};
\node at (17,2) {\lb};
\node at (17,1) {\lb};
\node at (18,2) {\lb};
\node at (19,3) {\lb};
\node at (22,0) {\lb};
\node at (23,1) {\lb};
\node at (23,0) {\lb};
\node at (27,0) {\lb};
\node at (27,1) {\lb};
\node at (27,2) {\lb};
\node at (29,3) {\lb};
\node at (30,4) {\lb};
\node at (31,5) {\lb};
\node at (31,4) {\lb};
\node at (31,3) {\lb};
\node at (31,2) {\lb};
\node at (31,1) {\lb};
\node at (32,2) {\lb};
\node at (33,3) {\lb};
\draw (2,1) -- (4,3) -- (4,1) -- (6,3);
\draw (13,0) -- (13,1);
\draw (15,2) -- (17,4) -- (17,1) -- (19,3);
\draw (22,0) -- (23,1) -- (23,0);
\draw (27,0) -- (27,2);
\draw (29,3) -- (31,5) -- (31,1) -- (33,3);
\node at (0,-.6) {$4$};
\node at (4,-.6) {$0$};
\node at (9,-.6) {$8$};
\node at (17,-.6) {$0$};
\node at (13,-.6) {$4$};
\node at (22,-.6) {$9$};
\node at (27,-.6) {$4$};
\node at (31,-.6) {$0$};
\node at (3,-1.6) {$k=5$};
\node at (14,-1.6) {$k=6$};
\node at (28,-1.6) {$k=7$};

\end{\tz}
\end{center}
\end{fig}
\end{minipage}
\bigskip

The analogue of Propositions \ref{Mzprop} and \ref{Uext} is as follows. It is proved using the exact sequences derived in Section \ref{extsec}.
\begin{prop} \label{cohshift} (a). For distinct $j_i\ge k$, $\ext_{A_1}(\zt,M_kz_{j_1}\cdots z_{j_r})$ is formed from\linebreak $\ext_{A_1}(\zt,\Sigma^{2^{{j_1}+2}+\cdots +2^{{j_r}+2}}M_k)$ by increasing filtrations by $r$ and extending to the left by $v_1^4$-periodicity.

(b). If $M$ is a $Q_0$-free $A_1$-module, then $\ext_{A_1}(\zt,U\ot M)$ is formed from $\ext_{A_1}(\zt,M)$
by increasing fitrations by 2 and extending to the left using $v_1^4$-periodicity.
\end{prop}

Analogously to \cite[Theorem 1.16]{DW}, we have
 the following remarkable duality result, where the group on the right hand side is the Pontryagin dual.
\begin{thm}\label{dualthm} There is an isomorphism of $ko_*$-modules $ko_*(K_2)\approx(ko^{*+6}K_2)^\vee$.\end{thm}
This is deduced from \cite[Corollary 9.3]{MR} similarly to the $ku$ proof in \cite{DG}.
The subtlety of the result is suggested by the observation that there is nothing like it for the $E_2$ pages. We anticipate illustrating it in subsequent work in which differentials and extensions are determined.

\def\line{\rule{.6in}{.6pt}}

\end{document}